\documentclass[12pt]{article}
\usepackage{fullpage}
\usepackage{color}
\usepackage{graphicx}
\usepackage{epstopdf}
\usepackage{amssymb,amsmath,amsthm,xfrac}

\usepackage[mathscr]{euscript}
\usepackage[all]{xy}
\usepackage{hyperref}
\usepackage[nameinlink,capitalize]{cleveref}

\definecolor{todocolor}{rgb}{1.0, 0.0, 0.5}

 \newtheorem{thm}{Theorem}[section]
 \newtheorem{cor}[thm]{Corollary}
 \newtheorem{lem}[thm]{Lemma}
 
 \theoremstyle{definition}
 
 \theoremstyle{remark}

 \numberwithin{equation}{section}

\begin{document}

\title{The Deltoid Curve and Triangle Transformations}
\author{Michael Q. Rieck }
\date{ \ } 
\maketitle
\begin{abstract}

Deltoid curves appear as consequences of certain procedures in triangle geometry. The best known of these is the construction based on Simson lines, described by Steiner. This is carefully related, in this article, to a less known construction. The standard deltoid in the complex plane and its tangent lines are principle objects of study in this report. It is known that each point in the interior of this curve is the orthocenter of a triangle with distinct vertices on the unit circle, whose product is one. (If instead the point is on the deltoid, then at least two of the vertices coalesce, resulting in a degenerate triangle.) 

When the vertices are all raised to some specified integer power, a new (possibly degenerate) triangle results. By varying the triangle, one may thus consider the map taking the original triangle's orthocenter to the resulting triangle's  orthocenter.  Such maps are the other principle objects of study here. The points that are mapped to the deltoid lie on easily described curves. By varying the power involved in the map, a pleasing family of curves results, which includes a trifolium curve. The points that are mapped instead to the origin are described as the points of intersection of certain tangents to the deltoid. 

\end{abstract}




\section{Introduction}\label{sec:intro}

Deltoid curves, also called tricuspids/tricuspoids, are easily described by rolling a circle inside a circle whose radius is three times bigger than that of the rolled circle. A point fixed relative to the rolled circle travels along a deltoid curve in the plane for which the larger circle is fixed (cf. \cite{F}, \cite{M}, and \cite{P}). All deltoids in a plane are of course equivalent in the sense that any one of them can be transformed into any other one by a combination of scaling, rotating and translating. It is convenient and sufficient throughout this paper to focus solely on the ``standard" deltoid in the Cartesian plane, which satisfies the equation 

\begin{equation}
(x^2 + y^2)^2 + 18 (x^2 + y^2) - 8 x^3 + 24 x y^2 - 27 \; = \; 0.
\end{equation}

\noindent When the plane is regarded as the complex plane, as will be the case throughout, the equation can be rewritten as 

\begin{equation}
z^2 \, \bar{z}^2  - 4(z^3 + \bar{z}^3 ) + 18 z \bar{z} - 27 \; = \; 0.
\end{equation}

\noindent This deltoid can also be expressed as the curve traced out by 

\begin{equation}
z \; = \; 2 e^{i \theta} + e^{-2 i \theta}
\end{equation}

\noindent as the real parameter $\theta$ ranges from say $-\pi$ to $\pi$ (cf. \cite{GPP}).  

\begin{lem}
For fixed $\theta$, the line segment connecting the two points $\pm 2 e^{i \theta} + e^{-2 i \theta}$ on the deltoid
is tangent to the deltoid at the point $e^{4 i \theta} + 2 e^{-2 i \theta}$. Also, this segment has slope $\tan \theta$ and length 4. 
\end{lem}
\begin{proof}
First we must establish that the points  $-2 e^{i \theta} + e^{-2 i \theta}$ and $e^{4 i \theta} + e^{-2 i \theta}$ are indeed on the deltoid; $2 e^{i \theta} + e^{-2 i \theta}$ certainly is. Replacing $\theta$ by $\theta + \pi$ in $2 e^{i \theta} + e^{-2 i \theta}$  yields $-2 e^{i \theta} + e^{-2 i \theta}$, so $-2 e^{i \theta} + e^{-2 i \theta}$ is on the deltoid.  Replacing $\theta$ by $-2 \theta$ in $2 e^{i \theta} + e^{-2 i \theta}$  yields $e^{4 i \theta} + 2 e^{-2 i \theta}$, so $e^{4 i \theta} + 2 e^{-2 i \theta}$ is on the deltoid. 

Next, we need to know that the points $\pm 2 e^{i \theta} + e^{-2 i \theta}$ and $e^{4 i \theta} + 2 e^{-2 i \theta}$ are collinear. $(e^{4 i \theta} + 2 e^{-2 i \theta}) - (\pm 2 e^{i \theta} + e^{-2 i \theta}) = e^{4 i \theta} + e^{-2 i \theta} \pm 2 e^{i \theta} = e^{i\theta} (e^{3 i \theta} + e^{-3 i \theta} \pm 2)  =  2 \, e^{i\theta} (\cos{3\theta} \pm 1) $. The tangent of the argument of this is just $\tan \theta$, regardless of which sign we use for ``$\pm$." So it is clear that the three points are collinear and lie on a line with slope $\tan \theta$. 

Now consider the tangent line to the deltoid at $z = x + i y = 2 e^{i \theta} + e^{-2 i \theta}$. This has slope  

$$\frac{dy}{dx} \; = \; \frac{\; \frac{dy}{d\theta} \; }{\; \frac{dx}{d\theta} \; } \; = \; \frac{\frac{d}{d\theta} \left[ 2 \sin \theta - \sin 2\theta \right] }{\frac{d}{d\theta} \left[ 2 \cos \theta + \cos 2\theta \right]} \; = \; \frac{2\cos\theta-2\cos 2\theta}{-2\sin\theta - 2\sin 2\theta}  \; = \; \frac{\cos 2\theta-\cos \theta}{\sin 2\theta + \sin \theta} \; = $$

$$\frac{2\cos^2\theta - \cos\theta - 1}{\sin\theta \, (2\cos t +1)} \; = \; \frac{(\cos\theta - 1)(2\cos\theta + 1)}{\sin\theta \, (2\cos t +1)} \; = \;  \frac{\cos\theta - 1}{\sin\theta} \; = \; \frac{-\sin\theta}{\cos\theta + 1} \; = \;  - \tan \frac{\theta}{2}.$$

\noindent By replacing $\theta$ with $-2\theta$, we see that the slope of the tangent line at $e^{4 i \theta} + 2 e^{-2 i \theta}$ is just $\tan \theta$, and therefore this must be the line connecting the two points $\pm 2 e^{i \theta} + e^{-2 i \theta}$.

Of course, the segment has length 4 since $(2 e^{i \theta} + e^{-2 i \theta}) - (-2 e^{i \theta} + e^{-2 i \theta}) = 4e^{i \theta}$. All that remains is to show that the point $e^{4 i \theta} + 2 e^{-2 i \theta}$ lies between the points $\pm 2 e^{i \theta} + e^{-2 i \theta}$.  But we can write $2 \, (e^{4 i \theta} + 2 e^{-2 i \theta}) = (1+\lambda)(2 e^{i \theta} + e^{-2 i \theta}) + (1-\lambda)(-2 e^{i \theta} + e^{-2 i \theta}) \; = \; 4 \lambda e^{i\theta} + 2 e^{-2 i \theta}$. Solving for $\lambda$, we get $\lambda = \frac{1}{2} ( e^{3 i \theta} + e^{-3 i \theta}) = \cos 3 \theta$. 
Since $|\lambda| \le 1$, it is now evident that $e^{4 i \theta} + 2 e^{-2 i \theta}$ lies between $\pm 2 e^{i \theta} + e^{-2 i \theta}$ (not necessarily strictly though).   

\end{proof}

The line segments described in the lemma are of particular importance. Because of their well-known connection with the Kakeya needle problem, they might be called {\it needle positions}, though they will just be referred to as {\it needles} here. With the above foundation, we will now proceed to present a few interesting and related phenomena, some of which are already known, but most of which appears to be new.  All of it is concerned with the geometry of a triangle, and very likely there are additional possible connections with the subject of triangle geometry that could be discovered. Actually, an already known connection involves rectangular circum-hyperbolas, though this will not be explored in this paper (cf. \cite{C}). It would also be desirable to better relate these topics to a certain three-dimensional geometry problem (cf. \cite{R}).

\section{Deltoids produced by triangle-related constructions}

\begin{figure}[ht]
  \centering
  \includegraphics[width=10cm]{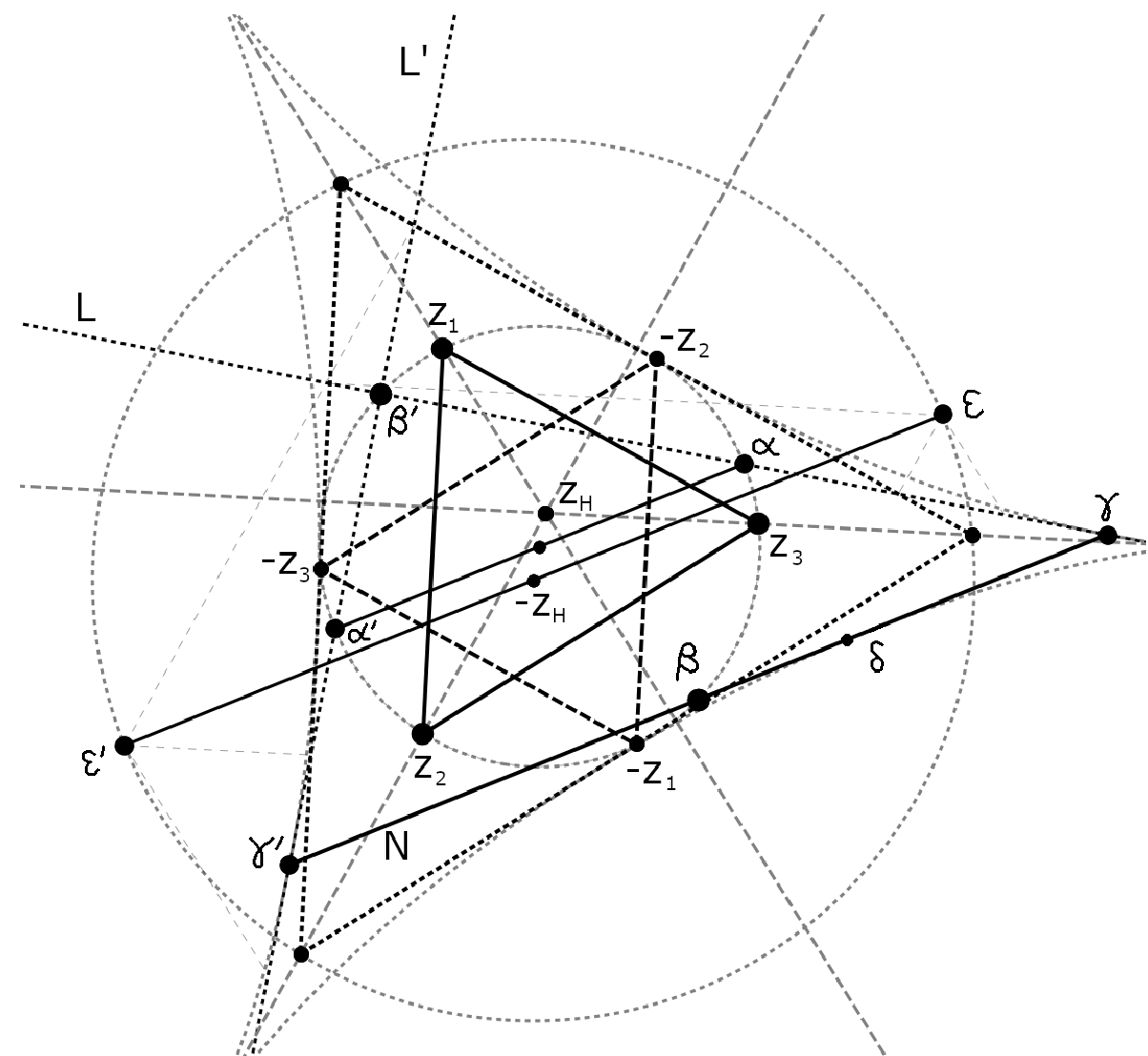} 
  \caption{Three triangles and a deltoid}
  \label{fig:triangles}
\end{figure}

Two of the better known deltoid constructions that result from studying triangle geometry are the Steiner deltoid and a deltoid that Kimberling describes in Chapter 6 of his book \cite{K}. Being unaware of any published proof of the existence of the latter, nor any published connection of it to the Steiner deltoid, such a connection will now be presented and proved. However, we will approach these constructions in reverse, by {\it starting} with the deltoid. Nevertheless, in this way we are able to establish the correctness of the two constructions. 

Beginning with the standard deltoid, and for any given real number $\theta$, let us identify the following points for discussion: $\alpha = e^{i \theta}$, $\alpha' = -\alpha = -e^{i \theta}$, $\beta = e^{-2i \theta}$, $\beta' = -\beta = -e^{-2i \theta}$,  $\gamma = e^{-2i \theta} + 2 e^{i \theta}$, $\gamma' = e^{-2i \theta} - 2 e^{i \theta}$, $\delta =  e^{4 i \theta} + 2 e^{-2 i \theta}$. As we know from Lemma 1.1, $\gamma$, $\gamma'$ and $\delta$ are collinear and lie on the deltoid, with $\delta$ between $\gamma$ and $\gamma'$, and this line is tangent to the deltoid at $\delta$ with slope $\tan \theta$. Referring to the segment of this line connecting $\gamma$ and $\gamma'$ as the ``needle" $N$, the midpoint of this needle is $\beta$ and this lies on the unit circle as well as on $N$. We now need to look at two other lines. Let $L$ be the line connecting $\gamma$ and $\beta'$, and let $L'$ be the line connecting $\gamma'$ and $\beta'$. So $L$ and $L'$ intersect at $\beta'$. The midpoint between $\gamma$ ($\gamma'$) and $\beta'$ is $\alpha$ ($\alpha'$), and so $\alpha$ ($\alpha'$) is on the line $L$ ($L'$). 

\begin{lem}
The lines $L$ and $L'$ are perpendicular. Moreover, $L$ ($L'$) is tangent to the deltoid at $\gamma$ ($\gamma'$). Thus, as $\theta$ varies, the pencil of lines $L$ ($L'$) has the deltoid as its envelope. 
\end{lem}
\begin{proof}
The slope of $L$ is the tangent of the argument of $(e^{-2i \theta} + 2 e^{i \theta}) - (-e^{-2i \theta})$, which is the tangent of the argument of $e^{-2i \theta} + e^{i \theta}$, which is 
$$\frac{-\sin 2\theta + \sin \theta}{\cos 2\theta + \cos \theta} \; = \; \frac{-2\cos\theta \sin \theta + \sin \theta}{2\cos^2 \theta - 1 + \cos \theta} \; = \; \frac{\sin \theta (1-2\cos\theta)}{-(1+\cos\theta)(1-2\cos\theta)} \; = \; - \tan \frac{\theta}{2}.$$
\noindent But in the proof of Lemma 1.1, it was shown that the tangent line to the deltoid at $\gamma$ has slope $\, -\tan \sfrac{\theta}{2}$, and so must be $L$. Similarly, it is straightforward to show that $L'$ and the tangent line to the deltoid at $\gamma'$ both have slope $\cot \sfrac{\theta}{2}$, and so must be the same line. Based on their slopes, it is clear that $L$ and $L'$ are perpendicular. There is a line $L$ ($L'$) for each $\gamma$ ($\gamma'$) on the deltoid, and it is clear that the deltoid is the envelope of this pencil of lines.  
 
\end{proof} 

Next fix three distinct complex numbers $z_1 = x_1 + i y_1$, $z_2 = x_2 + i y_2$ and $z_3 = x_3 + i y_3$ with $|z_1| = |z_2| = |z_3| = z_1 z_2 z_3 = 1,$ and regard these as the vertices of a triangle with the unit circle as its circumcircle. We will henceforth refer to such a triangle as being  ``amenable." Of course it will be a degenerate triangle if any two of $z_1$, $z_2$ and $z_3$ are equal, which is allowed. It will sometimes be helpful to select $\phi_1$, $\phi_2$ and $\phi_1$ so that $z_1 = e^{i\phi_1}$, $z_2 = e^{i\phi_2}$, $z_3 = e^{i\phi_3}$, and $\phi_1+\phi_2+\phi_3 = 0$. Let $z_H = z_1 + z_2 + z_3$. It is straightforward (cf. \cite{M}) to see that $z_H$ is the orthocenter of the triangle, and that $z_2 z_3 + z_3 z_1 + z_1 z_2 = \overline{z_H}$. Because two other (non-amenable) triangles need to be considered, the triangle with vertices $z_1$, $z_2$ and $z_3$ will be called the ``reference triangle." One of the other two triangles has $-z_1$, $-z_2$ and $-z_3$ as vertices, and so is the reflection of the reference triangle about its circumcenter. Let us call this the ``reflected triangle." We also need the triangle that has $z_1-z_2-z_3$, $-z_1+z_2-z_3$ and $-z_1-z_2+z_3$ as its vertices, which is the antimedial (anticomplementary) triangle of the reflected triangle. It can also be obtained via a homothetic transformation of the reference triangle, scaling by a factor of two, and using $z_H$ as the homothetic center. Thus, its orthocenter is also $z_H$. Let us call this triangle the ``large triangle." We are now ready to prove the claim about the deltoid construction in Chapter 6 of \cite{K}. 

\begin{thm}
With respect to the reflected triangle, the point $\beta$ is on its circumcircle, and the isogonal complement of $\beta$ is the point at infinity in the direction of the needle $N$ that passes through $\beta$, $\gamma$ and $\gamma'$.  Consider the line connecting $\beta$ to its isogonal complement (at infinity). As $\theta$ varies, the result is a pencil of lines whose envelope is the deltoid. 
\end{thm}

\begin{proof}
With respect to the reflected triangle, the interior angle bisector at the vertex $-z_3$ goes through a point of the unit circle midway between $-z_1$ and $-z_2$, so is either $\pm e^{i(\phi_1+\phi_2)/2} \, = \, \pm e^{-i\phi_3/2}$. If $+e^{-i\phi_3/2}$ then when the line through $-z_3$ and $\beta$ is reflected about this interior angle bisector, the resulting line passes through $-z_3$ and $e^{i[2(-\phi_3/2)-(-2\theta)]} = e^{i(2\theta-\phi_3)}$. But if $-e^{-i\phi_3/2}$ then in the previous computation, $\sfrac{\phi_3}{2}$ can be replaced with $\pi + \sfrac{\phi_3}{2}$ to yield the same answer, $e^{i(2\theta-\phi_3)}$. Let us now compute the slope of this resulting line through $-z_3$ and $e^{i(2\theta-\phi_3)}$. This is the tangent of the argument of $e^{i(2\theta-\phi_3)} + e^{i \phi_3}$, which equals 
$$\frac{\sin (2 \theta - \phi_3) + \sin \phi_3}{\cos (2 \theta - \phi_3) + \cos \phi_3} \; = \; \frac{-\cos 2\theta \sin \phi_3 + \sin 2\theta \cos \phi_3 + \sin \phi_3}{\cos 2\theta \cos \phi_3 + \sin 2\theta \sin \phi_3 + \cos \phi_3} \; = $$

$$\frac{\sin \phi_3 \, (1-\cos 2\theta) + \cos \phi_3 \sin 2\theta}{\cos \phi_3 \, (1+\cos 2\theta) + \sin \phi_3 \sin 2\theta} \; = \; \frac{\sin \theta \, (\sin \phi_3 \sin \theta + \cos \phi_3 \cos \theta)}{\cos \theta \, (\sin \phi_3 \sin \theta + \cos \phi_3 \cos \theta)} \; = \; \tan \theta.$$

\noindent This is the same as the slope of the needle $N$, and so $N$ is on the line that we just constructed. Now, since $\beta$ is on the circumcircle of the reflected triangle (the unit circle), this constructed line is aimed in the direction of the point at infinity that is the isogonal conjugate of $\beta$. Allowing $\theta$ to vary now, the constructed lines form a pencil of tangent lines to the deltoid, and it is clear that the deltoid is its envelope. 
\end{proof} 

We next turn our attention to the construction of a deltoid based on Simson lines, as developed by Steiner. This will here be seen to be related to the above construction. Our primary focus now is on the ``large triangle", {\it i.e.} the large dashed triangle in Figure 1, whose vertices are not labeled. First note that the nine-point circle of this triangle is the same as the circumcircle of the reference triangle (and the reflected triangle). 

\begin{thm}
With respect to the large triangle, let $\varepsilon$ and $\varepsilon' \, (= -\varepsilon)$ be the opposite ends of the diameter of the circumcircle of the large triangle that is parallel to the diameter of its nine-point circle, connecting $\alpha$ and $-\alpha$. Assume that when moving along these two diameters in the same direction, $\varepsilon$ and $\varepsilon'$ occur in the same order as do $\alpha$ and $\alpha'$.  Then the  Simson lines for $\varepsilon$ and $\varepsilon'$ are $L$ and $L'$, respectively. As $\theta$ varies, the lines $L$ ($L'$) form a pencil of lines whose envelope is the deltoid. 
\end{thm}

\begin{proof}
Recall that the large triangle is obtained from the reference triangle by a homothetic transformation, centered at $z_H$, using a scale factor of two. Its vertices are therefore $2 z_1 - z_H = z_1 - z_2 - z_3$, etc.  Since this transformation maps the reference triangle to the large triangle, it maps the reference triangle's circumcircle to the large triangle's circumcircle. It is also clear that $\varepsilon$ and $\varepsilon'$ are the images of $\alpha$ and $\alpha'$, respectively. Thus, $\varepsilon = 2 \alpha - z_H = 2 e^{i \theta} - z_H$ and $\varepsilon' = 2 \alpha' - z_H = -2 e^{i \theta} - z_H$. Consider the projection of $\varepsilon$ onto the sideline through the vertices $-z_1 + z_2 - z_3$ and $-z_1 - z_2 + z_3$. Let us denote this point as $w = u + i v$ ($u$ and $v$ real) for the moment. Letting $z_j = x_j + i y_j$ ($x_j$ and $y_j$ real; $j=1,2,3$), it is required that 

$$\frac{v - (-y_1+y_2-y_3)}{u - (-x_1+x_2-x_3)} \; = \;   \frac{y_2 - y_3}{x_2-x_3},$$

\noindent  so that $(y_3-y_2) u + (x_2-x_3) v \; = \; (x_3-x_2) y_1 + x_1 (y_2-y_3)$. The orthogonality means that it is further required that $(x_3-x_2)(u - 2 \cos \theta + x_H) + (y_3-y_2)(v - 2 \sin \theta + y_H) \; = \; 0$. Solving these two equations for $u$ and $v$ yields 

$$ u \; = \; \frac{(x_2 - x_3)^2}{1 - x_2 x_3-y_2 y_3} \, \cos \theta \; + \;  \frac{(x_2 - x_3)(y_2-y_3)}{1 - x_2 x_3-y_2 y_3} \, \sin \theta \; - \; x_1 \; = $$ 

$$ (1-x_2 x_3 + y_2 y_3)  \cos \theta  - (x_2 y_3 + x_3 y_2) \sin \theta  - x_1 \; = \; (1-x_1)  \cos \theta  + y_1 \sin \theta - x_1 $$ 

\noindent and 

$$ v \; = \; \frac{(x_2 - x_3)(y_2-y_3)}{1 - x_2 x_3-y_2 y_3} \, \cos \theta \; + \;  \frac{(y_2 - y_3)^2}{1 - x_2 x_3-y_2 y_3} \, \sin \theta \; - \; y_1 \; = $$ 

$$ -(x_2 y_3 + x_3 y_2)  \cos \theta + (1+x_2 x_3 - y_2 y_3) \sin \theta + y_1  \; =  \; y_1  \cos \theta + (1+x_1) \sin \theta - y_1. $$

Now we will show that $w$ is on the line $L$. $\alpha$ is on $L$, and $w - \alpha = w - e^{i \theta}$, so the slope of the line containing $w$ and $\alpha$ is 

$$\frac{v - \sin \theta}{u - \cos \theta} \; = \; \frac{y_1  \cos \theta + x_1 \sin \theta - y_1}{-x_1 \cos \theta  + y_1 \sin \theta - x_1} \; = \; \frac{-\sin \theta}{1+\cos{\theta}} \; = \; -\tan \frac{\theta}{2}.$$    

\noindent But in the proof of Lemma 2.1, it was shown that this is the slope of $L$. Therefore $w$ is on $L$, and we see that $L$ is the Simson line for $\varepsilon$. Similarly, $L'$ is the Simson line for $\varepsilon'$.  By Lemma 2.1, $L$ and $L'$ are tangent to the deltoid, and so as $\theta$ varies, $L$ ($L'$) sweeps out a pencil of lines whose envelope is the deltoid.\

\end{proof}

\section{Deltoid interior points as triangle orthocenters}

We begin this section by exploring needles and tangent lines for the deltoid, starting with the following technical lemma that is obvious from the graph of the deltoid. It is evident from its graph that the deltoid separates the rest of the plane into an ``inside" and an ``outside," and except at the three cusps of the deltoid, that a tangent line does not cut from one side to the other side at the point of tangency. 

One could observe here that if $x$ and $y$ in (1.1) are replaced respectively with $t x$ and $t y$, the discriminant of the resulting quartic polynomial in $t$ is a negative constant times $y^4 \, (3x^2-y^2)^4$. So generally, the quartic has only two (distinct) real roots. Assuming that $x$ and $y$ satisfy (1.1), then one of these roots is 1, and with a little effort, it can be established that the other root is negative. It is then safe to describe the inside (outside) of the deltoid as the totality of points $(t x, t y)$ for which $(x, y)$ satisfies (1.1) and $0 \le t < 1$ ($1 < t$). 

One way to establish the claim about the tangent lines would be to consider the {\it evolute} of the deltoid. This is the curve whose points are curvature centers for the deltoid. It is well known that the evolute of a deltoid is also a deltoid, three times bigger (in a linear sense) than the original deltoid, and oriented in the opposite direction. So the curvature centers are outside the circle of radius 3 and inside the circle of radius 9, while the original deltoid is inside the circle of radius 3, except at its cusps. This eliminates the possibility of any non-cusp  inflection points on the original deltoid.  

\begin{lem} 
Consider any point on the deltoid, and its tangent line. This tangent line contains a unique needle, which in turn contains the point, and which stays inside the deltoid, except at the points that are on the deltoid.  
\end{lem} 

\begin{proof} 
The point can clearly be written as $e^{4i\theta} + 2 e^{-2i\theta}$ for some real $\theta$. Lemma 1.1 makes it clear that this point is on a needle, which of course is on the tangent line to the deltoid at the point. The tangent line can be parameterized as $2 \lambda e^{i \theta} + e^{-2i \theta}$ with $\lambda$ ranging over the real numbers. In the proof of Lemma 1.1, we observed that  $e^{4 i \theta} + 2 e^{-2 i \theta} \; = \;  \pm 2 e^{i \theta} + e^{-2 i \theta} + 2 \, e^{i\theta} (\cos{3\theta} \pm 1)$, from which we see that $e^{4 i \theta} + 2 e^{-2 i \theta} \; = \; \lambda e^{i \theta} + e^{-2i \theta}$ with $\lambda = \cos 3\theta$. We also know that the points for which $\lambda = \pm 1$ are on the deltoid. 

If we now set $z = 2 \lambda e^{i \theta} + e^{-2i \theta}$  into (1.2), we obtain the equation 

$$4(1-\lambda^2)(1 - 2\lambda e^{3i \theta} + e^{6i \theta}) \; = \; 0.$$   

\noindent The only three solutions for $\lambda$ are the three that we have already identified, so the tangent line only intersects the deltoid at the corresponding three points. If $e^{4i\theta} + 2 e^{-2i\theta}$ is not a cusp of the deltoid, then the tangent line does not cut the deltoid at this point. The line can be seen to be non-tangent at $\pm 2 e^{i \theta} + e^{-2 i \theta}$, and so it cuts the deltoid at these two points. The point $e^{-2i \theta}$ is on the line and also inside the deltoid, except in the three special cases where it is on the deltoid. So except in a few special cases, we are able to say that $2 \lambda e^{i \theta} + e^{-2i \theta}$ is inside (outside)  the deltoid when $|\lambda| < 1$ ($|\lambda| > 1$). It is then very clear that the tangent line contains a unique needle, the one connecting $\pm 2 e^{i \theta} + e^{-2 i \theta}$, which of course contains $e^{4i\theta} + 2 e^{-2i\theta}$. In the special cases, $e^{4i\theta} + 2 e^{-2i\theta}$ can be seen to coalesce with one of the ends of the needle. These cases can be handled as limiting cases. 

\end{proof} 

\begin{lem} 

Two lines that are tangent to the deltoid intersect at a point on the deltoid or in its interior.  

\end{lem} 

\begin{proof}
Let's consider two tangent lines, $2 \lambda_1 e^{i \theta_1} + e^{-2i \theta_1}$ and $2 \lambda_2 e^{i \theta_2} + e^{-2i \theta_2}$, where $\theta_1$ and $\theta_2$ are fixed real numbers, but $\lambda_1$ and $\lambda_2$ ranges over all real numbers. By solving $2 \lambda_1 e^{i \theta_1} + e^{-2i \theta_1} \; = \; 2 \lambda_2 e^{i \theta_2} + e^{-2i \theta_2}$ for $\lambda_1$ and $\lambda_2$, we will of course locate the point of intersection of the two lines. This amounts to solving the following matrix equation: 

$$ 2 \left[ \begin{array}{cc} \cos \theta_1 & - \cos \theta_2 \\ \sin \theta_1 & \sin \theta_2 \end{array} \right] \; \left[ \begin{array}{c} \lambda_1 \\ \lambda_2 \end{array} \right] \; = \; \left[ \begin{array}{c}  \cos 2 \theta_2 - \cos 2 \theta_1 \\ \sin 2 \theta_1 - \sin 2 \theta_2  \end{array} \right]  \; {\ \atop .}$$

\noindent The solution to this is as follows: 

$$  \left[ \begin{array}{c} \lambda_1 \\ \lambda_2 \end{array} \right] \; = \; \left[ \begin{array}{c}  \cos( \theta_1 + 2 \theta_2)  \\  \cos( 2 \theta_1 +  \theta_2)  \end{array} \right]  \; {\ \atop .}$$

\vspace{2mm} 

\noindent In particular we see that  $|\lambda_1| \le 1$ and $|\lambda_2| \le 1$, indicating that the two tangent lines intersect at a point on the deltoid or in its interior.
\end{proof}

\begin{lem} 

Each of the three altitude lines for the reference (amenable) triangle are tangent lines of the deltoid. 

\end{lem} 

\begin{proof}

Consider again the reflected triangle (whose vertices are $-z_1$, $-z_2$ and $-z_3$). With respect to this triangle, the isogonal conjugate of $-z_1$ is the point at infinity in the direction given by the sideline connecting $-z_2$ and $-z_3$. The point $z_1$ is antipodal to $-z_1$ on the circumcircle, so its isogonal conjugate must be the point at infinity in the direction perpendicular to the sideline connecting $-z_2$ and $-z_3$. But this sideline is parallel to the sideline of the reference triangle connecting $z_2$ and $z_3$. So the construction in Theorem 2.3 assigns to the point $z_1$, the altitude line through $z_1$ for the reference triangle. But Theorem 2.3 indicates that this line is a tangent line for the deltoid. Similarly for the other two altitude lines. 

\end{proof} 

\begin{thm} 
With $z_1$, $z_2$ and $z_3$ as the vertices of a amenable triangle, its orthocenter $z_H$ satisfies $z_H = z_1 + z_2 + z_3$, and this point is on the deltoid or in its interior.

\end{thm} 

\begin{proof}

The fact that $z_H = z_1 + z_2 + z_3$ follows from $[(x_1+x_2+x_3) - x_1](x_2-x_3) + [(y_1+y_2+y_3) - y_1](y_2-y_3)  \; = \; (x_2+x_3)(x_2-x_3) +  (y_2+y_3)(y_2-y_3) \; = \; (x_2^2 + y_2^2) - (x_3^2+y_3^2) \; = \; 0.$ Now, by Lemma 3.3, each altitude line is a tangent line for the deltoid. By Lemma 3.2, these lines must intersect on the deltoid or in its interior. But of course $z_H$ is by definition this intersection point. 

\end{proof}

\begin{cor} 

Fix a real number $\theta$. Consider the needle whose slope is $\tan \theta$, that is, the needle parameterized by $2 \lambda e^{i \theta} + e^{-2i \theta}$ ($-1 \le \lambda \le 1$).  If the orthocenter $z_H$ of the triangle with vertices $z_1$, $z_2$ and $z_3$ is on the needle, then $z_H \; = \; 2 \lambda_0 e^{i \theta} + e^{-2i \theta}$ for some $\lambda_0$ with $-1 \le \lambda_0 \le 1$, and $\{ \, z_1, z_2, z_3 \, \} \; = \; \{ \, e^{-2 i \theta}$, $(\lambda_0 \pm i \sqrt{1-\lambda_0^2}) \, e^{i\theta} \, \}$ 

\end{cor}

\begin{proof} 

Just check that for these choices of $z_1$, $z_2$ and $z_3$, $|z_1| = |z_2| = |z_3| = z_1 z_2 z_3 = 1$ and $z_1 + z_2 + z_3 = 2 \lambda_0 e^{i \theta} + e^{-2i \theta}$.

\end{proof}

\section{ A family of transformations of triangles} 

Here and throughout the remainder of this paper, we continue to let  $z_1$, $z_2$ and $z_3$ be complex numbers satisfying 
$|z_1| = |z_2| = |z_3| = z_1 z_2 z_3 = 1$. We continue to regard these as the vertices of an amenable triangle, though this would be a degenerate triangle if any two of the vertices are the same, which is allowed. We continue to let  $z_H$ denote the orthocenter of the triangle, noting again that $z_H = z_1 + z_2 + z_3$ and that it is guaranteed to be on or inside the deltoid.

Fix an integer $n$. A function $p_n$ will be defined and investigated in this and in the next section of this paper. The domain and codomain of $p_n$ are the set of complex numbers $z$ on or inside the deltoid, that is, the set of  complex numbers $z$ satisfying   

\begin{equation}
z^2 \, \bar{z}^2  - 4(z^3 + \bar{z}^3 ) + 18 z \bar{z} - 27 \; \le \; 0.
\end{equation}

\noindent We will need the following fact before defining $p_n$. 

\begin{lem} 
If $z_H$ satisfies (4.1) when $z_H$ is substituted for $z$, then there exists a unique multi-set of numbers $\{ z_1, z_2, z_3 \}$ such that $|z_1| = |z_2| = |z_3| = z_1 z_2 z_3 = 1$ and $z_1 + z_2 + z_3 = z_H$. 
 \end{lem} 

\begin{proof}
As indicated previously, such numbers would necessarily also satisfy $z_2 z_3 + z_3 z_1 + z_1 z_2 \; = \; \overline{z_H}$, and so these numbers would be the roots of the following cubic equation:

\begin{equation}
z^3  \; - \; z_H \, z^2 \; + \; \overline{z_H} \, z \; - \; 1 \; = \; 0.
\end{equation}

\noindent Of course, the multi-set of roots of this is equation is unique, establishing the uniqueness of $\{ z_1, z_2, z_3 \}$. To establish the existence of a suitable $\{ z_1, z_2, z_3 \}$, we will cite Theorem 4 in \cite{M}, which uses the fact that the discriminant of the cubic polynomial here is the familiar $z_H^2 \, \overline{z_H}^{\, 2}  - 4(z_H^{\, 3} + \overline{z_H}^{\, 3} ) + 18 z_H \, \overline{z_H} - 27$, and which transforms the cubic, by substituting \linebreak  $(i w + 1)/(i w - 1)$ for $z$, to obtain a cubic in $w$ whose real roots are then investigated. The theorem asserts that for any complex number $z_H$, the discriminant is real, that at least one of the polynomial roots is on the unit circle, and that all three roots are on the unit circle if and only if the discriminant is negative. It is actually more accurate to say ``non-positive" here instead of ``negative" because the case when the discriminant is zero also results in all of the roots lying on the unit circle though now there will be a repeated root. 

\end{proof} 

We are now prepared to define the function $p_n$. Given a number $z_H$ on or inside the deltoid, let $\{ z_1, z_2, z_3 \}$ be the set of solutions to (4.2). These numbers can be regarded as the vertices of an amenable triangle. The set of numbers $\{ z_1^n, z_2^n, z_3^n \}$ also satisfies the same properties as the $\{ z_1, z_2, z_3 \}$, namely, $|z_1^n| = |z_2^n| = |z_2^n| = z_1^n z_2^n z_3^n = 1$. Regarding $\{ z_1^n, z_2^n, z_3^n \}$ as the vertex set for another amenable triangle, its orthocenter is simply  $ z_1^n + z_2^n + z_3^n$, and of course, this is on or inside the deltoid. $p_n(z_H)$ is now defined to equal  $z_1^n + z_2^n + z_3^n$. 

\begin{lem} 

A few examples of $p_n(z)$ are as follows

\begin{itemize}

\item[] \ \ \ \ $p_0(z) \; = \; 3$

\item[] \ \ \ \ $p_1(z) \; = \; z$

\item[] \ \ \ \ $p_2(z) \; = \; z^2 - 2 \bar{z}$

\item[] \ \ \ \ $p_3(z) \; = \; z^3 - 3 z \bar{z} + 3$

\item[] \ \ \ \ $p_4(z) \; = \; z^4 - 4 z^2 \bar{z} + 2 \bar{z}^2 + 4z$

\item[] \ \ \ \ $p_5(z) \; = \; z^5 - 5 z^3 \bar{z} + 5 z \bar{z}^2  + 5 z^2 - 5 \bar{z} $

\end{itemize} 

\noindent Additionally, the functions $p_n$ satisfy the following recurrence relation for $n \ge 4$: 

\begin{equation}
p_n(z) \; = \;  z \, p_{n-1}(z) \; - \; \bar{z} \, p_{n-2}(z) \; + \; p_{n-3}(z)
\end{equation}

\end{lem} 

\begin{proof}
From (4.2), we know that the elementary symmetric polynomials in $z_1, z_2, z_3$ have the following prescribed values: $\sigma_0 = 3$, $\sigma_1 = z_H$, $\sigma_2 = \overline{z_H}$, $\sigma_3 = 1$. Well-known identities of Newton relate the elementary symmetric polynomials and the power-sum elementary polynomials, conventionally denoted $p_n$. For instance, $p_2 = \sigma_1 p_1 - 2 \sigma_2$ and $p_3 = \sigma_1 p_2 - \sigma_2 p_1 + 3 \sigma_3$. Also, one of Newton's formulas states that if $n$ exceeds the number of indeterminates $N$ used in the polynomials, then 

$$ p_n \; = \; \sum_{j=n-N}^{n-1} (-1)^{n-1+j} \, e_{n-j} \, p_j $$ 

\noindent For our situation, $N = 3$ and $p_n \; = \; e_3 \, p_{n-3} - e_2 \, p_{n-2} + e_1 \, p_{n-1} \; = \; p_{n-3} - \bar{z} \, p_{n-2} + z \, p_{n-1}$.

\end{proof}

\begin{lem} 
For $n \ge 1$: 

\begin{equation}
p_n(z) \; \;  = \; \; n \cdot \,  \sum_{{\alpha, \beta, \gamma \ge 0} \atop {(\alpha+2\beta+3\gamma = n)}} 
\frac{(\alpha+\beta+\gamma-1)!}{\alpha! \, \beta! \, \gamma!} \; z^\alpha \left( - \bar{z} \right)^\beta
\end{equation}

\noindent Additionally, $p_n( e^{\pm 2\pi i / 3} z ) \; = \; e^{\pm 2\pi i n / 3} \, p_n(z)$.  

\end{lem} 

\begin{proof}
The summation formula is straightforward to check, by induction, using Lemma 4.2. Now, if $z$ is replaced with $e^{\pm 2\pi i / 3} z$, then the general term in the summation will be multiplied by $(e^{\pm 2\pi i / 3})^{\alpha - \beta}$. But $\alpha - \beta + 2n \; = \; 3\alpha + 3\beta + 6\gamma \; \equiv \; 0$ (mod 3). So $\alpha - \beta \; \equiv \; n$ (mod 3). Thus, $(e^{\pm 2\pi i / 3})^{\alpha - \beta} \; = \; e^{\pm 2\pi i n / 3}$.
\end{proof}

\section{Triangles with special ``powers"}

In this section, the functions $p_n$ introduced in the previous section will be investigated with an eye towards identifying points whose image is in some way special, and thereby also understand something about amenable triangles whose ``powers" are in some way special. We begin by looking at points $z$ inside the deltoid that are mapped by $p_n$ to points on the deltoid. Since only degenerate amenable triangles have an orthocenter on the deltoid, the points identified here will be the orthocenters of triangles with the property that at least two of the vertices have the same $n$-th power. This may not be a particularly interesting question to ask about the triangles, but the curves of points $z$ for which $p_n(z)$ is on the deltoid are rather interesting, as is the method for obtaining them. 

\begin{lem}
When $w^{-n} B +  w^{2n}$ is substituted for $z$ in $z^2 \, \bar{z}^2  - 4(z^3 + \bar{z}^3 ) + 18 z \bar{z} - 27$, assuming that $|w|$ = 1,  the result factors as $(B-2)(B+2) \, w^{-6n} (1 - B w^{3n} + w^{6n})^2$. 
\end{lem} 

\begin{proof} 
\ \\ 
$$\begin{array}{lll} 
z \bar{z} & \; = \; & (w^{-n} B + w^{2n})(w^n B + w^{-2n}) \; = \; B w^{-3n} + (1+B^2) + B w^{3n} \\ 
z^2 \bar{z}^2  & \; = \; & B^2 w^{-6n} + 2B(1+B^2) w^{-3n} + (1+4B^2+B^4) + 2B(1+B^2) w^{3n} + B^2 w^{6n} \\ 
z^3 + \bar{z}^3 & \; = \; & w^{-6n} + B(3+B^2) w^{-3n} +  6B^2 + B(3+B^2) w^{3n}  + w^{6n} 
\end{array}$$ 

\noindent Here are the coefficients of powers of $w$ in the expansion of $z^2 \, \bar{z}^2  - 4(z^3 + \bar{z}^3 ) + 18 z \bar{z} - 27$.  

\vspace{-3mm} 
$$\begin{array}{llll} 
w^{-6} & \; :  \; & (B-2)(B+2) & \hspace{9cm} \ \\
w^{-3} & \; :  \; & 2B(2-B)(B+2) & \ \\
w^{0} & \; :  \; & (B-2)(B+2)(B^2+2) & \ \\
w^{3} & \; :  \; &  2B(2-B)(B+2) & \ \\ 
w^{6} & \; :  \; & (B-2)(B+2) & \ 
\end{array}$$ 

\noindent The same coefficients occur in the expansion of $(B-2)(B+2) \, w^{-6n} (1 - B w^{3n} + w^{6n})^2$. 

\end{proof} 

\begin{lem}
Assuming that $n \ge 1$, that $A$ is real and that $|w| = 1$, \ $w^n \, p_n(w^2 + A/w) \, - \, w^{3n}$ is independent of $w$, and is a polynomial $q_n(A)$ in $A$ only. Moreover, $q_n(A) = A q_{n-1}(A) - q_{n-2}(A)$ for all $n \ge 3$. In fact, $q_n(A) \; = \;  (-i)^n L_n(i A)$ where $L_n$ is the $n$-th Lucas polynomial. Consequently, $p_n(w^2 + A/w) \; = \; w^{-n} q_n(A) + w^{2n} \; = \; (-i)^n w^{-n}  L_n(i A) + w^{2n}$.
\end{lem} 

\begin{proof} 
A direct check reveals that the claim is correct for $n = $ 1, 2 and 3. We now argue by induction for $n > 3$. Assume now that for a given positive $n > 3$, the claims are true for smaller values of $n$. Using the recurrence formula (Lemma 4.2), together with the induction hypothesis, we see $w^n \, p_n(w^2 + A/w) \, - \, w^{3n} \; = \; w^n \, [ \, (w^2 + A/w) \, p_{n-1}(w^2 + A/w) - (w^{-2} + A w) \, p_{n-2}(w^2 + A/w) +  p_{n-3}(w^2 + A/w) \; ]  - \, w^{3n} \; = \;  w^n \, [ \, (w^2 + A/w) (w^{1-n} q_{n-1}(A)+w^{2(n-1)})  - (w^{-2} + A w) (w^{2-n} q_{n-2}(A)+w^{2(n-2)})+  (w^{3-n} q_{n-3}(A)+w^{2(n-3)}) \; ]  - \, w^{3n} \; = \;  A \, q_{n-1}(A) - q_{n-2}(A) + [ \, q_{n-1}(A) - A \, q_{n-2}(A) + q_{n-3}(A)  \, ] \, w^3 \; = \; A \, q_{n-1}(A) - q_{n-2}(A) \; = \; A \, (-i)^{n-1} \, L_{n-1}(i A) - (-i)^{n-2} L_{n-2}(i A) \; = \;  (-i)^n \, [ \, (i A) \, L_{n-1}(i A) + L_{n-2}(i A) \, ] \; = \; (-i)^n L_n(i A)$.  Thus the claims are true for this particular value of $n$. By induction, the lemma is true. 

\end{proof}

\begin{lem}
Still assuming that $A$ is real, $4 - q_n(A)^2 \; = \; (-1)^n (A^2 - 4) \, F_n(i A)^2$, where $F_n$ is the $n-th$ Fibonacci polynomial. 
\end{lem} 

\begin{proof}
It is known that 

$$F_n(x) = \frac{(x+\sqrt{x^2+4})^n - (x-\sqrt{x^2+4})^n}{2^n \sqrt{x^2+4}} \; \hbox{and} \, L_n(x) = \frac{(x+\sqrt{x^2+4})^n + (x-\sqrt{x^2+4})^n}{2^n} {\ \atop .}$$ 

\noindent From these formulas, it is straightforward to deduce that $L_n(x)^2 - (x^2 + 4) F_n(x)^2 \; = \; 4 \, (-1)^n.$ Therefore, $4 - q_n(A)^2 \; = \; 4 - (-1)^n \, L_n(i A)^2 \; = \; 4 - (-1)^n \, [ \, (-A^2 +4) F_n(i A)^2 + 4(-1)^n \, ] \; = \; (-1)^n \, (A^2 - 4) \, F_n(i A)^2$. 

\end{proof}

\begin{lem}
When $p_n(z)$ is used in place of $z$ in $z^2 \, \bar{z}^2  - 4(z^3 + \bar{z}^3 ) + 18 z \bar{z} - 27$, and then $w^2 + A/w$ is substituted for $z$, where $A$ is real and $|w| = 1$,  the resulting expression is divisible by $(A^2 - 4) \, F_n(i A)^2$. So the resulting expression is identically zero when $A = \pm 2$ as well as when $i A$ is a root of $F_n$. 
\end{lem}

\begin{proof}
$p_n(z) = p_n(w^2 + A/ w) = w^{-n} q_n(A) + w^{2n}$, by Lemma 5.2, and when this is substituted for $z$ in $z^2 \, \bar{z}^2  - 4(z^3 + \bar{z}^3 ) + 18 z \bar{z} - 27$, the result equals $(q_n(A)^2 - 4) w^{-6n} (1 - q_n(A) w^{3n} + w^{6n})^2$, by Lemma 5.1. By Lemma 5.3, this is divisible by $(A^2 - 4) \, F_n(i A)^2$. The rest is then evident. \\
\end{proof} 

\begin{figure}[ht]
  \centering
  \includegraphics[width=6cm]{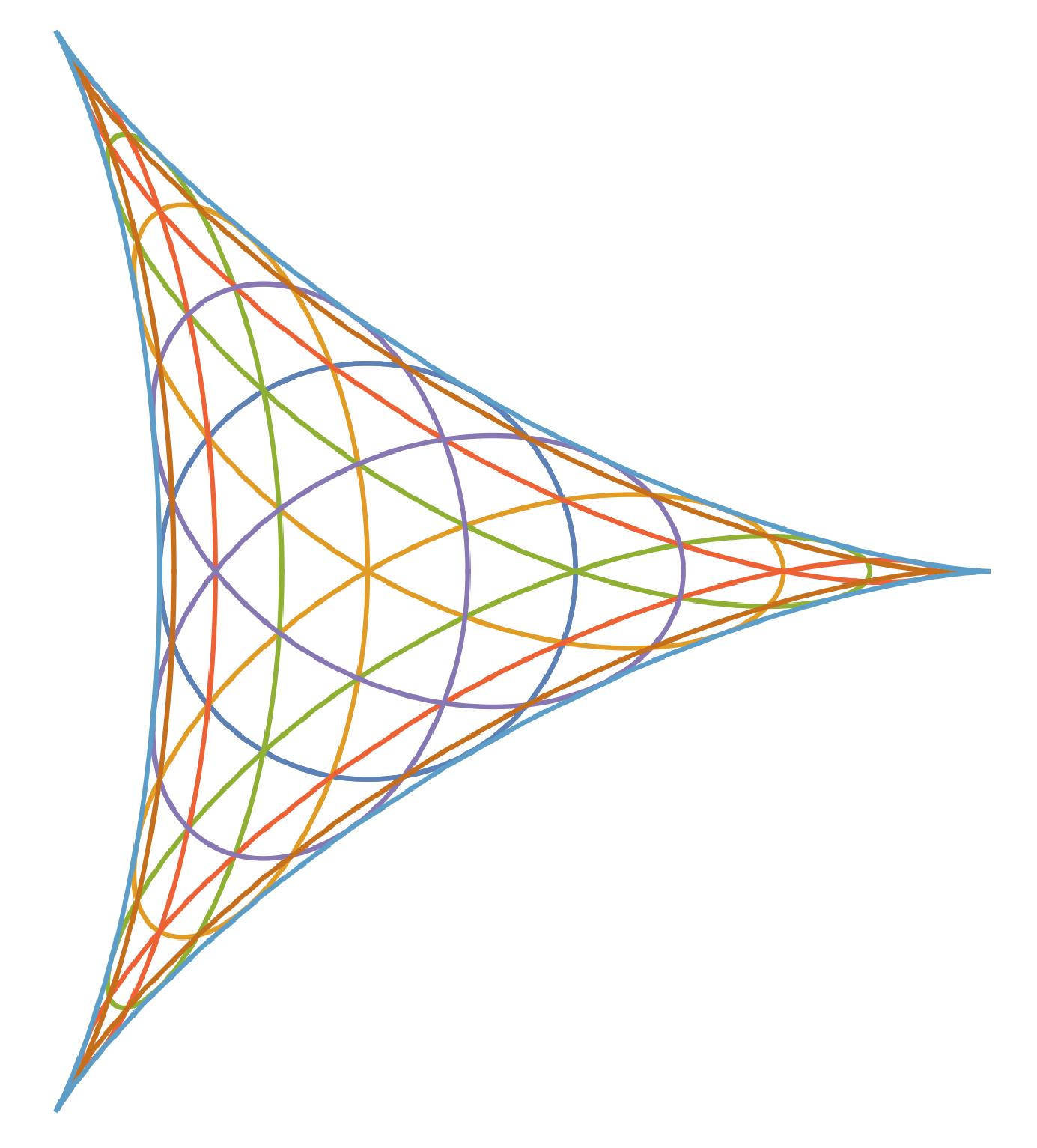} 
  \caption{Curves mapped to the deltoid via $p_{12}$.}
  \label{fig:curves}
\end{figure}

\begin{thm} 
Fix a positive integer $n$. If $n$ is even, let $A$ be one of the numbers $2\sin(j \pi / n)$ ($j = 0, 1, 2, ..., (n-2)/2)$. But if $n$ is odd, let $A$ be one of the numbers $2 \sin((2j+1) \pi / 2n)$ ($j = 0, 1, 2, ..., (n-3)/2$). The points on the curve that is parameterized by $A e^{i \theta} + e^{-2 i \theta}$ are mapped,  under the mapping $z \rightarrow p_n(z)$, to the deltoid. The deltoid is also mapped to the deltoid.   
\end{thm} 

\begin{proof}
From the results in \cite{HB}, $i A$ is a root of the Fibonacci polynomial $F_n$. This claim about the curve $A e^{i \theta} + e^{-2 i \theta}$ now follows immediately from Lemma 5.4, upon setting $w = e^{-i \theta}$. The deltoid is or course mapped to itself. $z_1$, $z_2$ and $z_3$ are not distinct when $z_H$ is on the deltoid, and so $z_1^n$, $z_2^n$ and $z_3^n$ are not distinct, and so $p_n(z_H)$ is also on the deltoid. 
\end{proof} 

It is worth noting that $A e^{i \theta} + e^{-2 i \theta}$ with $A = 1$ describes a trifolium curve. Figure 2 illustrates Theorem 5.5 for the case when $n = 12$. We will next see that the function $p_n$ maps needles to needles.

\begin{thm} 

Fix a real number $\theta$. Consider the needle whose slope is $\tan \theta$, that is, the needle parameterized by $2 \lambda e^{i \theta} + e^{-2i \theta}$ ($-1 \le \lambda \le 1$).  Fix also an integer $n$. The function $p_n$ maps the needle with slope $\tan \theta$ to the needle with slope $\tan n \theta$.  

\end{thm}

\begin{proof} 

By Corollary 3.5, we know that a point $2 \lambda e^{i \theta} + e^{-2i \theta}$ on the needle is the orthocenter for the triangle having the vertex set $\{ \, e^{-2 i \theta}, (\lambda \pm i \sqrt{1-\lambda^2}) \, e^{i\theta} \, \}$. Write $\lambda = \cos \psi$ for some real $\psi$. So the vertex set can be written as $\{ \, e^{-2 i \theta}, e^{i (\theta \pm \psi) } \, \}$. Raising these numbers to the $n$-th power, yields the vertex set $\{ \, e^{-2 i n \theta}, e^{i n (\theta \pm \psi) } \, \}$ for another triangle. This triangle has orthocenter  $e^{-2 i n \theta} + 2 \, e^{i n \theta} \cos n \psi $. This is evidently a point on the needle with slope $n \theta$. 

\end{proof}

\begin{figure}[ht]
  \centering
  \includegraphics[width=6cm]{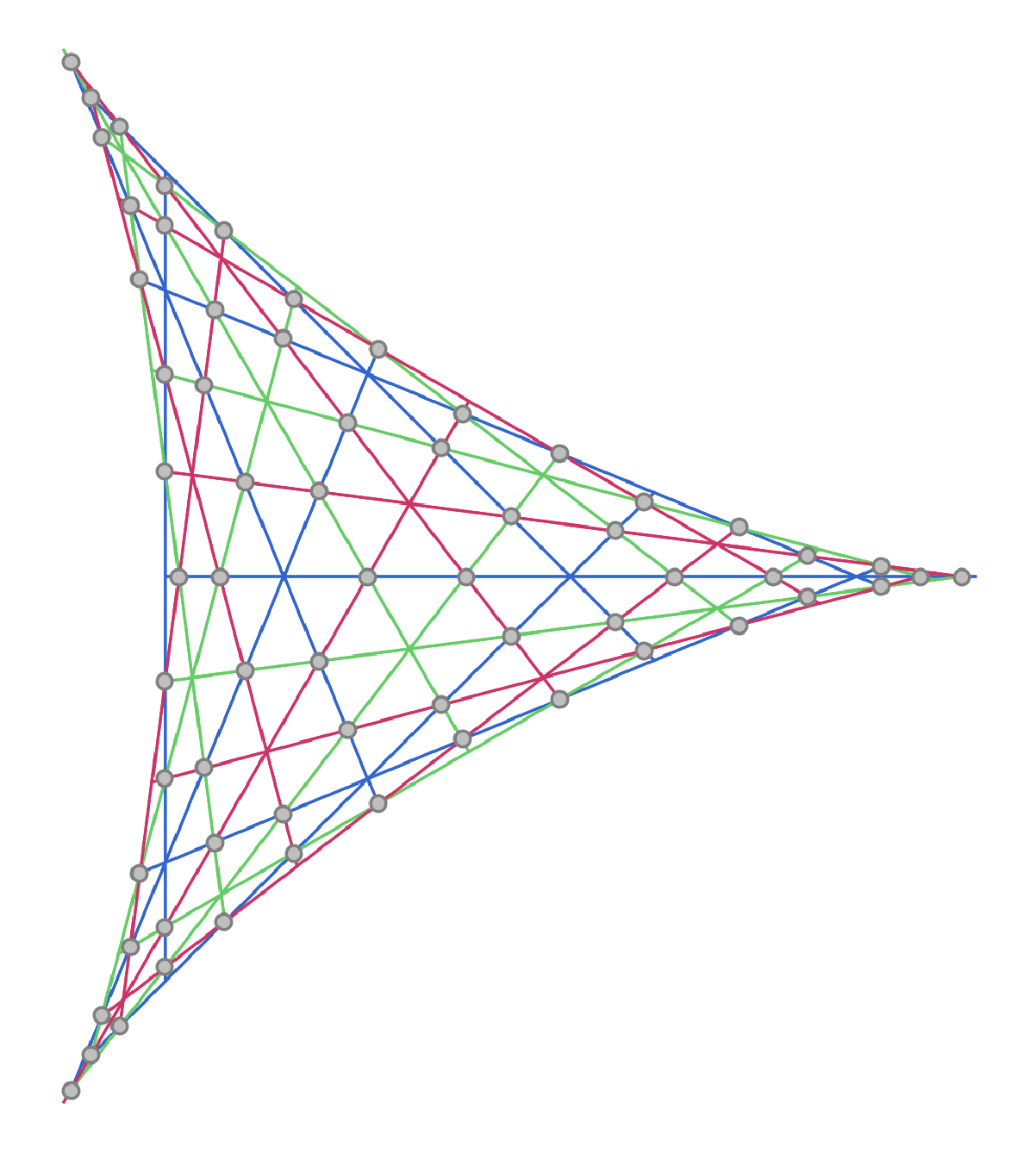} 
  \caption{Crossings of certain needles for $n = 8$.}
  \label{fig:crossings}
\end{figure}

We will close by considering the points that are mapped to zero by $p_n$. If $z_H = z_1+z_2+z_3 = 0$, then it can be reasoned that the triangle is equilateral, and in fact that $\{ z_1, z_2, z_3 \} \; = \; \{ 1, e^{2\pi i / 3}, e^{-2\pi i / 3} \}$. So at the level of the triangles, we are here asking about triangles whose ``$n$-th power" is this equilateral triangle. The following describes the orthocenters for these amenable triangles, {\it i.e.} all of the complex numbers $z$ for which $p_n(z) \; = \; 0$. Figure 3 illustrates this result when $n = 8$. 

\begin{thm} 
Fix a positive integer $n$. For $j_1, j_2 \in \{0, 1, 2, \cdots 3n-1 \}$ with $j_1 \not\equiv j_2$ (mod 3), let  $j_3 \in \{0, 1, 2, \cdots 3n-1 \}$ be such that $3n$ divides $j_1 + j_2 + j_3$. The three needles with slopes $\tan (-\pi j_1 / 3n)$, $\tan (-\pi j_2 / 3n)$ and $\tan (-\pi j_3 / 3n)$ are coincident, and meet at the point $e^{2\pi i j_1 / 3n} + e^{2\pi i j_2 / 3n} + e^{2\pi i j_3 / 3n}$. Moreover, $p_n(e^{2\pi i j_1 / 3n} + e^{2\pi i j_2 / 3n} + e^{2\pi i j_3 / 3n}) \; = \; 0$. In fact, the equation $p_n(z) = 0$ has $n^2$ solutions, all of which can be obtained in this manner. 
\end{thm} 

\begin{proof}

The three needles are described parametrically as $2 \lambda e^{i \theta_1} + e^{-2i \theta_1}$, $2 \lambda e^{i \theta_2} + e^{-2i \theta_2}$ and $2 \lambda e^{i \theta_3} + e^{-2i \theta_3}$, where $\theta_k \; = \; -\pi j_k / 3n$ ($k = 1, 2, 3$). The point $e^{-2i \theta_1} + e^{-2i \theta_2}  + e^{-2i \theta_3}$ is the intersection of these three needles. It is on the first needle because $e^{-2i \theta_1} + e^{-2i \theta_2}  + e^{-2i \theta_3} \; = \;   e^{-2i \theta_1} + e^{i [\theta_1 - (\theta_1+2\theta_2)]}  +  e^{i [\theta_1 + (\theta_1+2\theta_2)] } \; = \;   e^{-2i \theta_1} + 2 \, e^{i \theta_1} \cos (\theta_1+2\theta_2).$ Similarly for the second and third needles. 

Now consider the triple $\{z_1, z_2, z_3\}$ satisfying $|z_1| = |z_2| = |z_3| = z_1 z_2 z_3 = 1$ and $z_H \; = \; z_1 + z_2 + z_3 \; = \; e^{-2i \theta_1} + e^{-2i \theta_2}  + e^{-2i \theta_3}$. Since $z_H$ is on the first needle, by Corollary 3.5, one of the numbers $z_1$, $z_2$, $z_3$ must be $e^{-2i \theta_1}$. Similarly, one must be  $e^{-2i \theta_2}$, and one must be  $e^{-2i \theta_3}$. But these must be distinct because of the restictions on $j_1$, $j_2$ and $j_3$. So, $\{z_1, z_2, z_3\} \; = \;  \{ e^{-2i \theta_1}, e^{-2i \theta_2}, e^{-2i \theta_3} \} \; = \; \{ e^{2\pi i j_1 / 3n} , e^{2\pi i j_2 / 3n}, e^{2\pi i j_3 / 3n} \}$. So, $\{z_1^n, z_2^n, z_3^n\} \; = \; \{ e^{2\pi i j_1 / 3} , e^{2\pi i j_2 / 3}, e^{2\pi i j_3 / 3} \}\; = \; \{ 1, e^{2\pi i / 3}, e^{-2\pi i / 3} \} $. So, $p_n(z_H) \; = \; 0$.  

By varying $j_1$ and $j_2$ (and so also $j_3$) we obtain in this way $n^2$ distinct solutions to the equation $p_n(z) \; = \; 0$. Of course, because of the complex conjugate of $z$ appearing in (4.4), $p_n(z)$ is not a polynomial function of $z$. However, it is easy to see that there is a degree-$n$ polynomial function $P_n$ of {\it two} variables such that $p_n(z) \; = \; P_n(z, \bar{z})$, and that $\overline{p_n(z)} \; = \; p_n(\bar{z}) \; = \; P_n(\bar{z}, z)$.  So if $p_n(z) \; = \; 0$ then $P_n(z, \bar{z}) \; = \; P_n(\bar{z}, z) \; = \; 0$.  Using $w$ as a variable that is independent of $z$, the condition $p_n(z) \; = \; 0$ (for some $z$) implies that the system of equations  $P_n(z, w) \; = \; 0$ and $P_n(w, z) \; = \; 0$ has a common solution for $w$. The resultant polynomial in $z$, obtained by eliminating $w$ from the system of equations must then vanish too. However, by the theory of resultants, this polynomial in $z$ can have degree at most $n^2$, and so have at most $n^2$ roots. So there are at most $n^2$ solutions to the equation $p_n(z) \; = \; 0$. Since we have already identified $n^2$ solutions, these must be all of the solutions.     
\end{proof}


\end{document}